\documentclass[a4paper,12pt]{article}
\usepackage{amsmath}
\usepackage{amssymb}
\usepackage{amsthm}
\usepackage{verbatim}
\usepackage[affil-it]{authblk}
\usepackage{mathtools}

\DeclarePairedDelimiter\floor{\lfloor}{\rfloor}

\newtheorem{prop}{Proposition}

\newtheorem{thrm}{Theorem}[section]

\title{Non Existence of any Arithmetic Progression or Geometric Progression whose each term is a Palindrome}

\author{Sayak Chakrabarty%
  \thanks{Electronic address: \texttt{sayakc@cmi.ac.in}; }}
\affil{Chennai Mathematical Institute\\ Mathematics and Computer Science}

\author{Arhgya Datta%
  \thanks{Electronic address: \texttt{arghyad@cmi.ac.in}}}
\affil{Chennai Mathematical Institute\\ Mathematics and Theoretical Physics}

\begin{document}
\maketitle
\begin{abstract}
We investigate whether there exists an arithmetic progression or geometric progression consisting only palindromic numbers. In this paper we show that the answer to this question is NO. Given the first and final term we will also give an estimate for how large that AP could be and so for the GP given its first term and common ratio.
\end{abstract}

\section{Introduction}
A palindromic number or numeral palindrome is a number that remains the same when its digits are reversed. Like 16461, for example, it is "symmetrical".
In both cases we start with an initial term which is palindrome and supposing the entire sequence consists of only palindromic numbers we will arrive at a contradiction which shows that this sequence of palindrome eventually terminates. Hence in both cases we will give a bound for the maximum number of terms that the given progressions can possibly contain. Here we are only concerned about decimal representation. Here we assume that $gcd(a,2,3,5,11) = gcd(r,2,3,5,11) = 1$ where $a$ and $r$ are the first term and common ratio respectively of the GP.
\section{Arithmetic Progression}
\subsection{Non Existence of any Arithmetic Progression whose each term is Palindrome}
Supposing there is an AP whose all terms are palindromes, let the first term and the common difference be $a$ and $d$ respectively.We will denote the general $n^{th}$ term by $T_{n}$.The central idea will be to show that gaps between two consecutive palindromes can be arbitrarily large. Hence we cannot achieve an infinite AP through a constant common difference.\\
Let us start by getting hold of an algorithm from which it will be possible for us to write a palindrome which exactly succeeds a given one. We simply state here the procedure. Details and explanations can be found at [2].
We start by giving an example from which the procedure will be clear. The procedure will depend on the number of digits of the palindrome.\\
Example:\\
Start with a $5$(odd) digit palindrome, say $17371$. The number $3$ which acts as a pivot will only get replaced by $4$ as to make sure that the new resulting palindrome will be as small as possible such that it exactly succeeds $17371$. So the palindrome right next will be $17471$. The reader may like to notice that in general for palindromes with odd number of digits we are required to keep all the digits same which comes exactly in the first half before the pivot number(assuming it is not $9$).
This along with increasing the pivot digit by 1 will make sure that the number will be
smallest such palindrome which exceeds the given one. The 9 case has to be treated differently.

In case the given palindrome had a pivot which is 9 , say for example 3459543 , we observe the palindrome has to be bigger than at least 3460000 ( because in order to keep it smallest the starting 2 digits should be 34) and since it is palindrome we know it would be a reflection across the pivot number so it ends with 43 and to get the last two digits 43 from 3459543 we have to increase 3459543 by adding the least possible number to it and which can be seen to be possible by changing that 9 to 0 which will give rise to carry 1 in both the succeeded and preceded neighbouring digits of 9. So the palindrome right next to 3459543 would be 3460643. Same can be generalized in case of any palindromes having odd number of digits with occurrence of 9 in pivot. In this paper we do not need the algorithm for even digit palindromes and the reason for the same will be discussed later.\\
\begin{prop}
There exists no AP whose every term is palindrome
\end{prop}
\begin{proof} From the above context it can be readily be seen that if $(a_{1}...a_{n}a_{0}a_{n}...a_{1})$ be a $2n+1$ digit palindrome with $a_{0}$ as the pivot then the palindrome next to it will have the first and last $n-1$ terms same. In case $a_{0}$ were 9, then $a_{n}$ would have increased by 1. So the difference between these two would as least be $10^{n}$. So \\
$ l  i  m_{n\to\infty}$  $(z_{n+1}-z_{n}) = \infty$ where $z_{n+1},z_{n}$ are the $(n+1)^{th}$ and $(n)^{th}$ palindromes respectively.

Fact: Given an AP with first term $a$ and common difference $d$, there does not exist $M$ such that $T_{n}$ has only even digits for $n\geq M$. This is because consecutive powers of 10's increases unboundedly as $10^{p+1}-10^{p} = 9.10^{p}$ while $d$ is constant. So $10^{p} > d$ as $p\to\infty$ \\
So the given AP is forced to take values with odd number of digits infinitely often. Since $d$ remains constant it is a contradiction to the fact that difference between two consecutive palindromes having odd number of digits can be arbitrarily large.\\
\end{proof}
\subsection{Estimation of the largest AP with palindromes given its first and last terms}
Given the first and last term $a$,$l$ respectively in order to attain the largest AP, we need to choose the common difference $d$ suitably, in particular we need to have a lower bound for $d$. Let ...

\section{Geometric Progression}
\subsection{Non Existence of any Geometric Progression whose each term is Palindrome}
Suppose we are given a geometric progression whose first term $a$ is a palindrome and common ratio $r$ is rational which consists of only palindromes.\\
Claim: $r$ must be an integer.
\begin{proof}
Suppose the distinct prime factors of $a$ are $p_{1},..p_{k}$ and the distinct prime factors of the denominator of $r$ be $q_{1},..q_{m}$. But since the power of $r$ goes on increasing $ar^{s}$ will not remain integer for some $s \geq S$ for some $S \in \mathbb{N}$. So $r$ must be an integer.
\end{proof}
Assume number of digits of $a$ is $L$ and the number of digits of $r$ is $R$. Introduce a parameter $\lambda \in \mathbb{N}$. We would like to know the smallest term of the GP which has at-least $\lambda.L$ digits. If $t_{k}$ denotes the $k^{th}$ term for this GP, the smallest such $k$ such that $T_{k} = ar^{k} \geq \lambda L$ since a $L$ digit palindrome has size at-least $\sim 10^{L}$\\
$\implies k \geq \dfrac{\lambda. L}{log a + log r} \sim \dfrac{\lambda L}{L+ R -2}$  since x has $1 + \floor*{\log{x}}$ digits\\
\\
\begin{prop}
Given $R > L$ $\nexists$ a GP whose every term is palindrome
\end{prop}
\begin{proof}
Let us denote $\mathcal{P_{L}}$ the set of all $L$ digit palindromes.\\
$\mathcal{P_{L}}(q):= \{ n \in \mathcal{P_{L}} : n \equiv 0(\mod q) \}$\\ Proposition 4.2 of [1] asserts that if $(q, g(g^{2} − 1)) = 1$, where $g$ is the base and in our case $g=10$. Then for$L \geq  10 + 2q^{2}.log q$ the following asymptotic formula holds:\\
\\
$\|\mathcal{P_{L}}(q)\| = \dfrac{\|\mathcal{P_{L}}\|}{q} + O(\dfrac{\|\mathcal{P_{L}}\|}{q} exp(-\dfrac{L}{2q^{2}}))$\\
Here we obtain a nontrivial bound on $\mathcal{P_{L}}(q)$ without any restrictions on the size or the arithmetic structure of q.\\
So in our context\\
$\mathcal{P_{\lambda L}}(t_{k}):= \{ n \in \mathcal{P_{\lambda L}} : n \equiv 0(\mod ar^{k}) \}$\\
Using the above result \\
\\
$\|\mathcal{P_{\lambda L}}(t_{k})\| \leq \dfrac{C.10^{\frac{\lambda L}{2}}}{a r^{\frac{\lambda l}{L + R-2}}}$ \\
\\
$\implies \|\mathcal{P_{\lambda L}}(t_{k})\| \leq  \dfrac{C}{a} {\dfrac{10^{\frac{\lambda L}{2}}}{r^\frac{\lambda L}{L + R-2}}}$\\
\\
$\implies \|\mathcal{P_{\lambda L}}(t_{k})\| \leq  \dfrac{C}{a} ({\dfrac{10^{\frac{ L}{2}}}{r^\frac{ L}{L + R-2}}})^{\lambda}$\\
\\
Since $R > L$ , $R-1 > \dfrac{L+R-2}{2}$ \\ 
\\
which is equivalent to $r > 10^{\frac{L+R-2}{2}}$ by taking logarithm in base 10.\\
$\implies \ {\dfrac{10^{\frac{ L}{2}}}{r^\frac{ L}{L + R-2}}}$ ($= \alpha$ say) $< 1$ (can be verified by taking $L^{th}$ roots)\\
\\
Now given any $\epsilon > 0$, since the parameter $\lambda$ can be arbitrarily large $\exists \lambda_{0}$ such that $\dfrac{C}{a}. \alpha ^{\lambda} < \epsilon$ for all $\lambda \geq \lambda_{0}$\\
\\
So \\
$\|\mathcal{P_{\lambda_{0} L}}(t_{k})\| \leq \epsilon $ .\\
Contradiction.

\end{proof}
\subsection{There exists no GP whose every term is palindrome}
\begin{proof}
Suppose we are given a GP whose first term $a$ is a palindrome and common ratio $r$ is an integer.(Here we assume gcd(a,2,3,5,11) = gcd(r,2,3,5,11) = 1) We can choose an integer $B$ such that $r^{B}$ has more digits than $a$. Therefore we look at the subsequence:\\
$ {ar^{B},ar^{2B},ar^{3B},ar^{4B}....}$ which obeys every condition for proposition 2.\\
Hence this GP will cannot consist of only palindromes. Since this is a subsequence of the original GP, it also cannot contain infinitely many palindromes.
\end{proof}

\subsection{Given the initial term and common ratio , the size of the largest GP whose every term is palindrome}

\section*{Acknowledgement}
The authors thank Prof. Anirban Mukhopadhyay  for raising the question of  investigation and the authors are grateful to the referees W. Banks, D. Hart and M. Sakata [1], for the proposition 4.2 taken from their paper.

\end{document}